\documentclass[12pt]{amsart}
\usepackage{extsizes}
\usepackage{hyperref}
\usepackage{comment}
\hypersetup{
    colorlinks=true,
    linkcolor=blue,
    filecolor=magenta,      
    urlcolor=cyan,
}

\usepackage{graphicx}
\usepackage{multicol}
\newtheorem{theorem}{Theorem}[section]
\newtheorem{lemma}[theorem]{Lemma}

\theoremstyle{definition}
\newtheorem{definition}[theorem]{Definition}
\newtheorem{example}[theorem]{Example}
\newtheorem*{rem}{Remark}

\usepackage{comment}

\DeclareMathOperator{\suppo}{supp}
\DeclareMathOperator{\cod}{codim}

\DeclareMathOperator{\ima}{Im}
\DeclareMathOperator{\Hess}{Hess}
\DeclareMathOperator{\He}{H}
\DeclareMathOperator{\Te}{T}

\DeclareMathOperator{\conv}{conv}
\DeclareMathOperator{\mv}{MV}
\DeclareMathOperator{\res}{Res}
\usepackage{comment}
\usepackage{graphicx}
\usepackage{subcaption}

\usepackage{amssymb}

\usepackage{geometry}

\theoremstyle{remark}

\numberwithin{equation}{section}

\begin{document}

\title{Bivariate systems of polynomial equations with roots of high multiplicity.
}


\author{I. Nikitin 
}

\footnotetext{
              National Research University Higher School of Economics, Russian Federation, Department of Mathematics, 6 Usacheva st, Moscow 119048; \\
              \email{i.s.nikitin@yandex.ru (more preferable), isnikitin@edu.hse.ru}           
}


\maketitle

\begin{abstract}
Given a bivariate system of polynomial equations with fixed support sets $A, B$ it is natural to ask which multiplicities its solutions can have. We prove that there exists a system with a solution of multiplicity $i$ for all $i$ in the range $\{0,1,...,|A|-|\conv(A)\ominus B|-1\}$, where $A\ominus B$ is the set of all integral vectors that shift B to a subset of $A$. As an application, we classify all pairs $(A, B)$ such that the system supported at $(A, B)$ does not have a solution of multiplicity higher than $2$.
\\

Keywords$\colon$Differential geometry, Algebraic geometry, Weierstrass points, Newton Polytope
\end{abstract}

\section{Introduction}
\label{intro}
For any projective algebraic variety $X\subset\mathbb{CP}^n$ one can define the dual variety $X^{*}\subset(\mathbb{CP}^n)^{*}$ as the set of hyperplanes $H$ tangent to the smooth part $X_{sm}$ of X, i.e. such that $H\cap X_{sm}$ is not smooth. Typically $X^{*}$ is a hypersurface, and if not, then $X$ is called dual defective. The study of all exceptional defective cases when $\cod X^{*}\geq1$ is a classical problem \cite{E}.

The notion of the dual variety is used in the study of discriminants. Consider a finite subset $A$ in the monomial lattice $\mathbb{Z}^{n}$ and the set $\mathbb{C}^{A}$ of Laurent polynomials whose monomials belong to $A$. The computation of the discriminant of a polynomial with the support $A$, i.e. the set $\triangle_{A}=\{f\in\mathbb{C}^{A}|\ df=f=0\ \text{has a solution}\}$ is strongly related to the aforementioned problem about dual varieties through the following almost tautological observation. The discriminant variety $\triangle_{A}=0$ is projectively dual to the toric variety $X_{A}$ \cite{GKZ} (p.3), that is defined as follows$\colon X_A$ is the closure of the image of $(\mathbb{C}^{*})^n$ by the so-called monomial map $m_{A}$ 
$$
m_{A}\colon(\mathbb{C}^{*})^{n}\rightarrow\mathbb{CP}^{A}
$$
$$
x\mapsto [\ldots\colon x^{a_{i}} \colon\ldots].
$$

We may consider subsets of $(\mathbb{CP}^{n})^*$ consisting of hyperplanes $H$ such that $H\cap X_{sm}$ has prescribed singularities. In that case, the  dual projective  space  admits a natural  stratification where  every stratum corresponds to a collection of singularities of specific type. For $X_{A}$ the stratum with the lowest positive codimension is $\triangle_{A}$. The problem of the classification of dual defective $X_{A}$ has been recently solved in \cite{AET} and \cite{KA}. Some further strata of the singularity stratification are studied e.g. in \cite{DRR} and \cite{AET}. 

In more general setting, if the variety in $(\mathbb{C}^{*})^n$ is given by a system of Laurent polynomials
$$
f_{1}=f_{2}=\ldots=f_{d}=0
$$
with given support sets $A=(A_{1},\ldots, A_{d}),\ A_{i}\subset\mathbb{Z}^{n}$, then we define the mixed $A$-discriminant as the algebraic closure in $C=\mathbb{C}^{A_1}\oplus\ldots\oplus\mathbb{C}^{A_d}$ of the set of all systems that have a degenerate root, i.e. a point such that $f_{i}(p)=0$, $i=0,\ldots, d$ and $df_{i}$ are linearly dependent. In the same manner denote by $C_{k}$ the stratum of all systems that have a root of Milnor number $k-1$ (e.g. if $d=n$, $C_{k}$ is the set of all systems that have a root of multiplicity $k$). Natural problems that arise here are$\colon$
\begin{itemize}
    \item Classify all $A$ for which $C_{k}$ is empty for some low $k$.
    \item What properties should $A$ satisfy to make $C_{k}$ of expected codimension?
    \item Estimate the length of maximal possible chain of non empty strata $$C_{k}, C_{k-1}, \ldots, C_{0}.$$
\end{itemize}   For $k=2, n=d$ these problems were solved in \cite{BN} and \cite{EA}, which lead to new results on the Galois groups of general systems of polynomial equations. From the results of the work \cite{GK} it follows that the maximal possible $k$ such that $C_k$ is non empty can be bounded from above by a function depending on the total number of monomials in $A$ but not on their degrees. Also, in the recent work \cite{DRM} it was shown that the only smooth polygons $A$ such that curves $\{f=0\},\, f \in \mathbb{C}^A$, may have at most $A_1$-singularities, are equivalent to the 2–simplex or a unit square. In this paper we study the problem about the continuous chain of strata in the case $n=d=2$ and arbitrary $k$.

Let us formulate the main result of the paper. Let $A, B$ be finite subsets of $\mathbb{Z}^2$ such that $A, B$ cannot be shifted to the same proper sublattice and such that neither $A$ nor $B$ lies on a segment. In this paper we show that systems supported at $(A, B)$ can have roots of every multiplicity from $1$ to $|A|-|\conv(A)\ominus B|-1$,
where $\conv(A)$ is the convex hull of $A$ and $X\ominus Y$ is the set of all integral vectors that shift $Y$ to a subset of $X$.

In order to demonstrate the main idea that we will use in the proof of the main result, we are going to discuss the following classical theorem$\colon$
\begin{theorem}\label{ts}
Let $P=c_{1}t^{a_{1}}+c_{2}t^{a_{2}}+\ldots+c_{k}t^{a_{k}}$ be a $k$-sparse polynomial in a single complex variable. Then for any $l\in\{0, 1,\ldots, k-1\}$ there exists a set of coefficients $\{c_{1}, c_{2},\ldots, c_{k}\}$ such that $P(t)$ has a non-zero root of multiplicity $l$.
\end{theorem}

The proof is based on a simple linear-algebraic argument. We will use the same idea to obtain the main result of the paper so let us give the proof.

\begin{proof}[\underline{Proof of Theorem \ref{ts}}]
Let $P(t)=c_{1}t^{a_{1}}+c_{2}t^{a_{2}}+\ldots+c_{k}t^{a_{k}}$. We show that there exists a set of coefficients such that $t=1$ is a root of desirable multiplicity $l\in 0, 1,\ldots, k-1$. To avoid zero terms we multiply $P(t)$ by the monomial $t^k$. Since $P(t)$ is a polynomial in one variable the condition ``$t=1$ is a root of multiplicity $k-1$" is equivalent to existence of a non-zero solution of the following system of linear equations$\colon$
\[
P(1)=P'(1)=P''(1)=\ldots=P^{(i)}(1)=\ldots=P^{(k-2)}(1)=0,
\]
or in matrix notation, 
\[
\begin{bmatrix}
1 & 1 & 1 & \dots & 1 & 1\\
a_{1} & a_{2} & a_{3} & \dots & a_{k-1} & a_{k} \\
\dots  & \dots  & \dots  & \dots & \dots & \ldots \\
\frac{a_{1}!}{(a_{1}-k+2)!} & \frac{a_{2}!}{(a_{2}-k+2)!} & \frac{a_{3}!}{(a_{3}-k+2)!} & \dots & \frac{a_{k-1}!}{(a_{k-1}-k+2)!} & \frac{a_{k}!}{(a_{k}-k+2)!}
\end{bmatrix}
\begin{bmatrix}
c_{1} \\ c_{2} \\ \dots \\ c_{k} 
\end{bmatrix}
=0.
\]

Applying elementary transformations to the system, we get the following$\colon$

\[
\begin{bmatrix}
1 & 1 & 1 & \dots & 1 & 1\\
a_{1} & a_{2} & a_{3} & \dots & a_{k-1} & a_{k}\\
a^2_{1} & a^2_{2} & a^2_{3} & \dots & a^2_{k-1} & a^{2}_{k}\\
\dots  & \dots  & \dots  & \dots & \dots & \ldots  \\
a^{k-2}_{1} & a^{k-2}_{2} & a^{k-2}_{3} & \dots & a^{k-2}_{k-1} & a^{k-2}_{k} 
\end{bmatrix}
\begin{bmatrix}
c_{1} \\ c_{2} \\ \dots\\ \dots \\ c_{k} 
\end{bmatrix}=0.
\]
The maximal minor of this matrix is the Vandermonde determinant, and since all $a_{k}$ are distinct, the matrix has maximal rank. Since the matrix has maximal rank, solutions of the first $l$ rows is a vector space of dimension $k-l$. Then, the vector space corresponding to solutions of the first $l+1$ rows is a proper subspace of the solutions corresponding to the first $l$ rows. Now to obtain a set of coefficients for which $t=1$ is a root of multiplicity $l$, we can take any solution of the first $l$ equations of the matrix which does not satisfy equations number $l+1,\ldots, k-1$. 
\end{proof}
\begin{rem}
Also, from this proof it follows that for any given set $A\subset\mathbb{Z}$, such that $|A|=k$, we cannot get a root of multiplicity $k$ in the complex torus $\mathbb{C}^{*}$ for any set of coefficients. Assume that such set of coefficients exists, then we have a square system of linear equations whose coefficient matrix has maximal rank and we conclude that this system has only the trivial solution.
\end{rem}

For further discussions, let us introduce some notation. Recall that for each $a=(a_{1}, a_{2})\in\mathbb{Z}^{2}$ we use the expression $z^{a}$ to denote the monomial $z_{1}^{a_{1}}z_{2}^{a_{2}}$, where $z_{i}\in\mathbb{C}^{*}$ and for each $A\subset\mathbb{Z}^{2}$ we denote by $\mathbb{C}^{A}$ the space of Laurent polynomials $\sum_{a\in A}c_{a}z^a$ supported at $A$. For any polynomial $f$ the convex hull of its support is called the Newton polygon of $f$. For a pair of two subsets $S=(A_{1},  A_{2})\subset\mathbb{Z}^{2}$, we have $\mathbb{C}^{S}=\mathbb{C}^{A_{1}}\oplus\mathbb{C}^{A_{2}}$. We identify each  $f=(f_{1}, f_{2})\in\mathbb{C}^{S}$ with the corresponding system of polynomial equations $f_{1}=f_{2}=0$. For any two subsets $A, B$ of\ $\mathbb{Z}^{2}$ their Minkowski sum $A+B$ is defined as follows$\colon A+B=\{a+b|a\in A, b\in B\}\subset\mathbb{Z}^{2}$. We say that a finite subset $A\subset\mathbb{Z}^2$ is convex if it can be represented in the form $A=\mathbb{Z}^2\cap\bar{A}$, for some convex $\bar{A}\subset\mathbb{R}^{2}$. 

By $\mathbb{CP}^{A}$ we denote the projective space whose homogeneous coordinates are enumerated by $A$. We denote the zero set of an ideal or a set of polynomials by $\mathcal{Z}(\ldots)$. We say that a curve $\gamma$ is supported at $A\subset\mathbb{Z}^2$ if there exists $f\in\mathbb{C}^A$ such that $\gamma=\mathcal{Z}(f)$. The number of points in the set $A$ is denoted by $|A|$. The mixed volume of two convex polygons $A, B$ is denoted by $\mv(A, B)$.

Given a pair of subsets $S=(A_{1}, A_{2})$, we consider the subset $M_{S}$ of natural numbers consisting of multiplicities that a root of a system supported at $S$ may have. More precisely$\colon$
$$
M_{S}=\{m\in\mathbb{N}\mid \text{there exists a system}\ f\in\mathbb{C}^{S}\ \text{with a root of multiplicity}\ m\}.
$$
\begin{figure}[h!]
  \centering
     \includegraphics[width=0.6\linewidth]{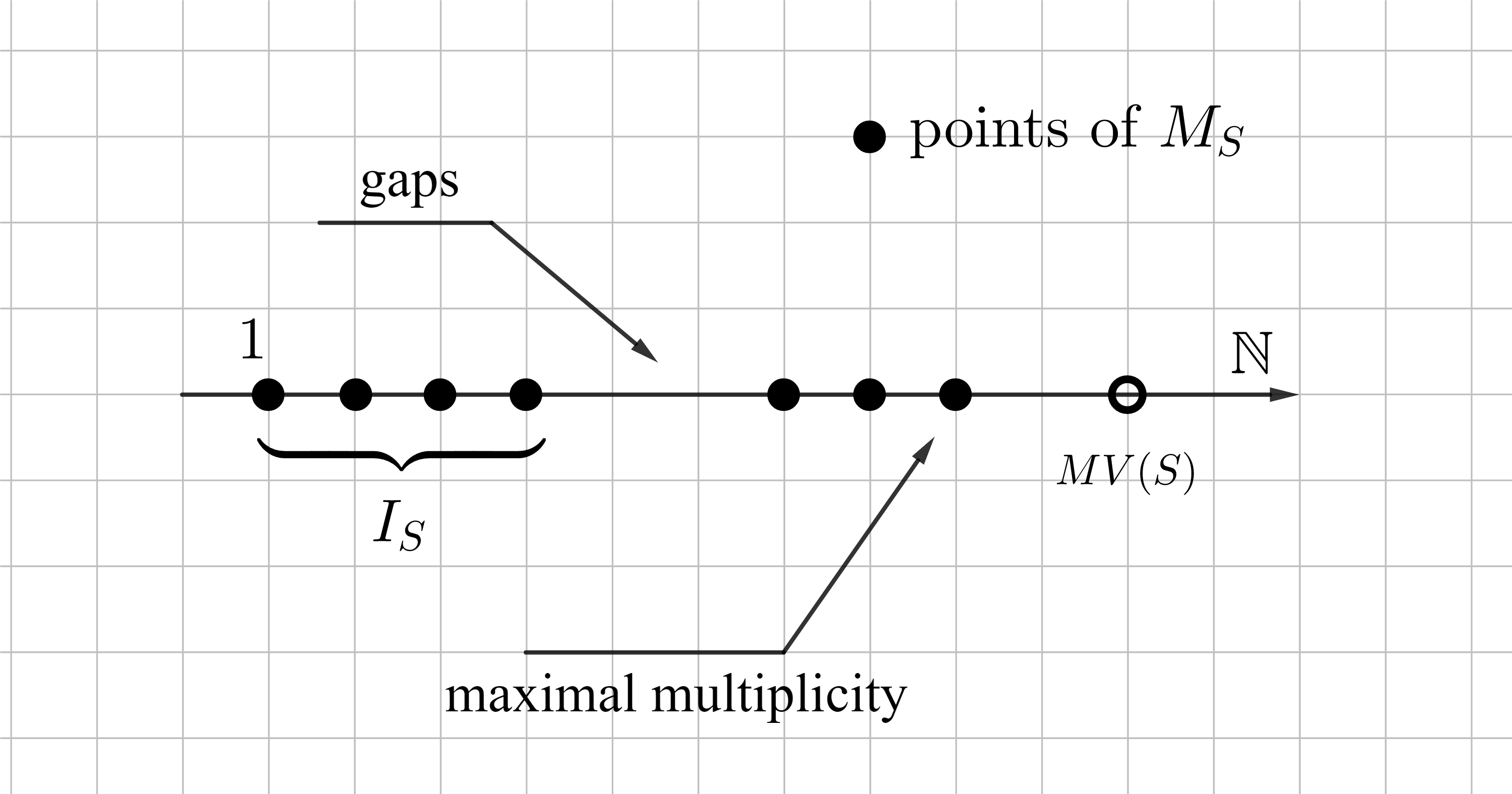}
    \caption{Description of $M_{S}$}
\label{y}
\end{figure}

The simplest natural problems that arise here are (see Figure \ref{y})$\colon$
\begin{enumerate}
\item Estimate $|I_{S}|$, where $I_{S}\subset M_{S}$ is the maximal set of consecutive integers in $M_{S}$ that is starting at 1.  
\item Find the maximal possible multiplicity $\max(M_{S})$.
\item Find integers which are less or equal than the mixed volume of $S$ but do not lie in $M_{S}$, i.e. find all the multiplicities that the system supported at $S$ cannot have.
\end{enumerate}

In this paper we focus on the first problem and give examples which show that the answers to (2) and (3) are nontrivial. The main result is$\colon$
\begin{theorem}\label{th}
Let $A, B\subset\mathbb{Z}^{2}$ such that $A$ and $B$ cannot be shifted to the same proper sublattice and neither $A$ nor $B$ lies on a segment. Then for almost all points on a generic curve $\gamma$ supported at $B$, there exists a set of osculating curves $g_{i}$ supported at $A$, such that $g_{i}$ intersects $\gamma$ with multiplicity $i$ for any $i\leq|A|-|\conv(A)\ominus B|-1.$
\end{theorem}

\begin{rem}Let us discuss what happens when we relax some of the conditions of the theorem.
\begin{enumerate}
    \item Suppose that $A$ lies on a segment, and $B$ is a generic subset of $\mathbb{Z}^2$. Without loss of generality assume that the segment containing $A$ is horizontal (one can achieve that by applying a suitable monomial change of variables). Then any system supported at $S=(A, B)$ is of the form $p(x)=q(x,y)=0$. For any set of coefficients the multiplicity of a root $(x_{0}, y_{0})$ of the system $p(x)=q(x,y)=0$ is equal to $m_{1}m_{2}$, where $m_{1}$ is the multiplicity of $x_{0}$ in $p(x)=0$ and $m_{2}$ in the multiplicity of $y_{0}$ in $q(x_{0}, y)=0$. Both numbers can be estimated using Theorem \ref{ts}.
    \item Suppose that $(A, B)$ can be shifted to the same proper sublattice of $\mathbb{Z}^2$. In this situation there exists a new pair of subsets $(A_{0}, B_{0})$ and a lattice embedding $i\colon\mathbb{Z}^2\rightarrow\mathbb{Z}^2$ such that $(A_{0}, B_{0})$ cannot be shifted to the same proper sublattice and $A=i(A_{0}), B=i(B_{0})$. Since the dual covering of complex tori $i^{*}$ preserves multiplicities we have the similar range $|I_{A, B}|\geq|A_{0}|-|\conv(A_{0})\ominus B_{0}|-1$.
\end{enumerate}
\end{rem}
\begin{rem}
Note that applying the theorem to $(A, B)$ and $(B, A)$ may give different estimates. In particular, if $|A|\leq|B|$, then we get the following inequality $$|I_{A,B}|\geq\begin{cases}
|A|-2\ \text{if}\ A=B\\
|A|-1\ \text{if}\ A\neq B.
\end{cases}$$\end{rem} In our notation we have the following chain of inequalities$\colon$
$$
|A|-|\conv(A)\ominus B|-1\leq|I_{S}|\leq\max(M_{S})\leq\mv(S).
$$

Each of these inequalities may be strict for some $S$ or may have an equality case, below we give examples to demonstrate this. Examples \ref{ex3} and \ref{ex10} are given for the first and second inequality. The last inequality also can be strict and it follows from the mentioned result in \cite{GK}. We also give an example \ref{exim} based on the result from Section \ref{s1}. We say that $S$ is {\it left-degenerate} if $$\max(|A|-|\conv(A)\ominus B|-1, |B|-|\conv(B)\ominus A|-1)=|I_{S}|$$ and {\it right-degenerate} if $|I_{S}|=\max(M_{S})$. 
The proof of the main result is given in Section \ref{thr12}. We mentioned that the result may be considered as a generalisation of Theorem \ref{ts}. Indeed, set $B=(0,1)\cup\{(a_{i}, 0)\}_{i=1}^{k}$ where $a_{i}$ are distinct positive integers and $A=\{(0, 0), (1, 0), (0, 1)\}$. Then $A$ does not lie on any segment and by the formula given in Theorem \ref{th} we get that the upper bound $|I_{S}|\geq|B|-|B\ominus A|-1=k+1-1-1=k-1$, and it is the same multiplicity that was given via Theorem \ref{ts}.

As an application of the above results we obtained the following theorem$\colon$
\begin{theorem}\label{sing}
A system supported at a pair $(A,B)$ of Newton polygons can not have a root of multiplicity 3 if and only if $\mv(A, B)\leq2$, i.e. if a system supported at $(A, B)$ can have at least $3$ roots we can obtain a root of multiplicity $3$.
\end{theorem}
 We will prove this theorem in the last section. The analogue of this theorem for multiplicity 2 is proved in \cite{EA} for any dimension. Whether theorem \ref{sing} is valid for any dimension is unknown. Note that the analogue of Theorem \ref{sing} for multiplicities higher than 3 does not hold already in dimension 2, as we see in the following example.

 \begin{example}\label{exim}
 There exists a pair of Newton polygons such that a system supported at it can not have a root of multiplicity 4 and whose mixed volume is equal to 4. Take $A=\{(0, 0), (1, 0), (0, 1)\}$ and $B=\{(0, 1), (3, 0), (4, 0)\}$. If a system supported at $(A, B)$ has a root of multiplicity $\geq3$ then a curve supported at $B$ has an inflection point. Any trinomial curve supported at $B$ has two inflection points in the origin and in the complex torus. Without loss of generality we can assume that the inflection point in the complex torus is located at $(1, 1)$. Computing the resultant of the system 
 $$
 \begin{cases}
 F=y-1-a(x-1)=0\\
 G=y-b x^3-(1-b) x^4=0,
 \end{cases}
 $$
 we obtain $R(y)=\res(F, G, y)=(x-1) (-1+a-x-x^2-x^3+bx^3)$, and it is easy to show that the system $R(1)=R'(1)=R''(1)=0$ has no solutions in $(a, b)$. Then $(1, 1)$ is a non-degenerate inflection point as well as $(0, 0)$.
 \end{example}

\begin{example}\label{ex3}For $n\geq3$ the pair $S=(A_{n}, A_{n})$ where $A_{n}$ is $\conv((0, 0), (n, 0), (0, n))$, is right non-left degenerate. We have
$$
|A_{3}|-|A_{3}\ominus A_{3}|-1<|I_{S}|=\max(M_{S})=\mv(S).
$$
This example shows that Theorem \ref{th} in general does not give us the length of $I_{S}$ but only an estimate. Indeed, according to Bezout's theorem, a system supported at $S=(A_{n}, A_{n})$ has at most $n^2$ isolated roots. Let us assume that $n\geq3$. Theorem \ref{th} gives the following estimate$\colon$
$$
|A_{n}|-2=n(n+3)/2-1\leq I_{S}\leq n^2.
$$

We show that $S$ is not left-degenerate, for example, if $n=3$ then the largest possible multiplicity is $9$ and $n(n+3)/2-1=8$, but $I_{S}=\max(M_{S})=9$. For $n>3$ we can also show that there exists a system that has a root of multiplicity $n(n+3)/2=(|A_{n}|-2)+1$. For instance, we take
$$
u_{k,l} =\prod_{1}^{n}h_{i}\in\mathbb{C}^{A_{n}},
$$
$$
v_{k,l}=\prod_{1}^{l}h_{i}^{a_{i}}+H_{k+1}\in\mathbb{C}^{A_{n}}, a_1 + a_2 + \ldots + a_l = k, l\leq k\leq n-1,
$$
where $h_{i}$ are pairwise independent linear forms and $H_{k+1}$ is a homogeneous form of degree $k+1$ that does not vanish on lines \{$h_{i}=0$\}. Polynomials $u_{k,l}$ and $v_{k,l}$ are well-defined since $l\leq k$, $k+1\leq n$. This system has a root of multiplicity $l(k+1)+(n-l)k=nk+l$ at zero. Now, for odd $n>3$ set $k=(n+3)/2$ and $l=0$, then the system $u_{k,l}=v_{k,l}=0$ has a root of multiplicity $nk+l=n(n+3)/2=(|A_{n}|-2)+1$. In a similar way, for even $n$ we set $k=n/2+1$, $l=n/2$. These computations imply that $S$ is not left-degenerate for $n>3$ and
$$
|A_{n}|-|A_{n}\ominus A_{n}|-1<|I_{S}|\leq\max(M_{S})=\mv(S).
$$
We do not discuss whether the second inequality is sharp.
\end{example}

\begin{example}\label{ex10}
Now let us show that the tuple $A=\{(0, 0), (1, 0), (0, 1), (0, 2),\ldots , (0, n)\}$, \\$B=\{(0, 0), (0, 1), (1, 0), (2, 0)\}$ is left-degenerate and is not right-degenerate for  $n=2l+1$, $l>1$. Then we have a tuple $S=(A, B)$ such that 
$$
|A|-|A\ominus B|-1=|I_{S}|<\max(M_{S})=\mv(S).
$$
Consider two plane curves $\gamma_{2}$, $\gamma_{n}$ given by the polynomials of the form$\colon$
$$
\gamma_{n}\colon c_{n}y-p_{n}(x)=0, p_{n}\in\mathbb{C}^{B};
$$
$$
\gamma_{2}\colon c_{2}x-q_{2}(y)=0, q_{2}\in\mathbb{C}^{A}.
$$
By Bernstein–Kouchnirenko theorem, such a system can have at most $2n$ isolated roots, but we cannot get a system that has a root of multiplicity $n+2$. Indeed, if one of $c_{n}$ or $c_{2}$ is zero we can glue together only even number of roots or number of roots that is less than $n+1$, if $c_{n}, c_{2}\neq0$ then without loss of generality we may assume that $(0, 0)$ is the root of multiplicity $k$. Then after substitution, we will have the equation$\colon$
$$
x=q_{2}(p_{n}(x)),
$$
which has a root of multiplicity $k$ at $0$. Therefore $q_{2}(p_{n}(x))-x=a_{k} x^{k}+h.o.t.$, $a_{k}\neq0$. We may assume that the leading coefficient of $q_{2}$ is 1. Expanding $q_{2}(p_{n}(x))-x$ we get $\colon$
$$
(a_{n}x^n+a_{n-1}x^{n-1}+\ldots+a_{1}x+a_{0})^2+a(a_{n}x^n+a_{n-1}x^{n-1}+\ldots+a_{1}x+a_{0})+b-x=0.
$$
The system $x-q_{2}(y)=y-p_{n}(x)=0$ has a root of multiplicity $k+1$ at zero if and only if the coefficients at the monomials $x^0, x^1,\ldots,x^{k}$ vanish and the $(k+1)$-th does not. Thus we have the following system$\colon$
$$
\begin{cases}
x^0\colon a_{0}^2+a a_{0}+b=0\Rightarrow b=-(a_{0}^2+a a_{0})\\
x^1\colon 2a_{1}a_{0}+a a_{1}=1\Rightarrow a_{1}=1/(a+2a_{0})\\
x^2\colon 2a_{2}a_{0}+a_{1}^2+aa_{2}=0\Rightarrow a_{2}=-1/(a+2a_{0})^3\\
\ldots\\
x^k, k\text{ is odd} \colon 2(a_{k}a_{0}+a_{k-1}a_{1}+\ldots+a_{i}a_{k-i})+aa_{k}=0\\
x^k, k\text{ is even} \colon 2(a_{k}a_{0}+a_{k-1}a_{1}+\ldots+a_{i}a_{k-i})+a_{k/2}^2+aa_{k}=0.
\end{cases}
$$
We claim that the solution is of the form $$a_{k}=(-1)^{k-1}\phi_{k}/(a+2a_{0})^{2k-1}, k=0,1,\ldots,n,\ \phi_{k}>0,$$ and the coefficient at the monomial $x^{n+1}$ never vanishes. Indeed, for $k=1$ this is true, for odd $k$ we have $a_{k}=-2(a_{k-1}a_{1}+\ldots+a_{i}a_{k-i})/(a+2a_{0})$. Consider 
$$a_{i}a_{k-i}=\frac{(-1)^{i-1}\phi_{i}(-1)^{k-i-1}\phi_{k-i}}{(a+2a_{0})^{2i-1}(a+2a_{0})^{2(k-i)-1}}=\frac{(-1)^{i-1+k-i-1}\phi_{i}\phi_{k-i}}{(a+2a_{0})^{2i-1+2(k-i)-1}}=\frac{(-1)^{k-2}}{(a+2a_{0})^{2k-2}}\phi_{i}\phi_{k-i},$$
then we have
$$a_{k}=-2\frac{(-1)^{k-2}}{(a+2a_{0})^{2k-2}}\underbrace{(\phi_{k-1}\phi_{1}+\ldots+\phi_{i}\phi_{k-i})}_{\phi_{k}}\frac{1}{(a+2a_{0})}=(-1)^{k-1}\frac{2\phi_{k}}{(a+2a_{0})^{2k-1}}.$$
Computations for even $k$ are similar. The $(n+1)$-st coefficient is given by the formula
$$
2(a_{n}a_{1}+a_{n-1}a_{2}+\ldots+a_{i}a_{n+1-i}),
$$
and using the formula $a_{k}=(-1)^{k-1}\phi_{k}/(a+2a_{0})^{2k-1}$ we show that this coefficient is not zero.
Thus we cannot obtain the root of multiplicity $n+2$ and the theorem gives us $I_{S}\geq n+2-1=n+1$, hence $S$ is left-degenerate. On the other hand, we can glue together all the solutions, hence $S$ is not right-degenerate. We conclude that the tuple is left-degenerate and not right-degenerate as desired. For example, the case $n=3$ gives us $M_{S}=\{1, 2, 3, 4, 6\}$, as can easily be computed directly, and we see that $M_{S}$ has a gap in 5.
\end{example}
\section{\bf Proof of the main result}\label{thr12}
In this section we prove the main result of the paper.
\begin{definition}
The monomial map supported at $A$ is given by the rule

$$
m_{A}\colon(\mathbb{C}^{*})^{n}\rightarrow\mathbb{CP}^{A}
$$
$$
x\mapsto [\ldots\colon x^{a_{i}} \colon\ldots].
$$
\end{definition}
The map is well defined on $(\mathbb{C}^{*})^n$. We denote $m_{A}((\mathbb{C}^{*})^n)$ by $\mathbb{T}$. Let $S=(A, B)$ be a tuple of subsets of $\mathbb{Z}^{2}$, and $g=\sum_{a\in A}c_{a}x^{a}$ be a polynomial supported at $A$. By $\pi_{g}$ we denote the hyperplane in $\mathbb{CP}^{A}$ given by the linear form $\sum_{a\in A} c_{a}z_{a}$, i.e. $\pi_{g}$ has the same coefficients as $g$ has.

\begin{lemma}\label{l1}
Let $A, B$ be subsets of $\mathbb{Z}^{2}$ such that $B$ does not lie on a segment. Then for a generic $f\in\mathbb{C}^B$ and $g\in\mathbb{C}^{A}$ the following conditions are equivalent$\colon$
\begin{enumerate}
    \item $m_{A}(\mathcal{Z}(f))$ is contained in $\pi_{g}$;
    \item $\exists c\in\mathbb{C}[x^{\pm1}, y^{\pm1}]$ such that $g=cf$.
\end{enumerate}
\end{lemma}
\begin{proof}[\underline{Proof}]
    Let us denote $m_{A}(\mathcal{Z}(f))$ by $F$. Assume that $F$ lies in $\pi_{g}$. Take $p\in F$, then there exists $x\in(\mathbb{C}^{*})^{2}$ such that $p=m_{A}(x)$. Since $p\in\pi_{g}$ we have $\sum c_{a}x^{a}=0$ therefore $x\in \mathcal{Z}(g)$ and $\mathcal{Z}(f)\subset \mathcal{Z}(g)\Rightarrow$ by Hilbert nullstellensatz for Laurent polynomials there exists $c\in\mathbb{C}[x^{\pm1}, y^{\pm1}]$ and $k\in\mathbb{N}$ such that $g^{k}=cf$. By assumption $B$ does not lie in any segment then, any generic $f\in\mathbb{C}^{B}$ is irreducible by \cite{Kh}. Since the ring of Laurent polynomials is UFD, $g^{k}=cf$ imply that there exists $\bar{c}\in\mathbb{C}[x^{\pm1}, y^{\pm1}]$ such that $g=\bar{c}f$. For the other implication, let $g=cf$, then the vanishing set of $f$ vanishes $g$ as well hence $\mathcal{Z}(f)\subset \mathcal{Z}(g)$. Therefore $F\subset\pi_{g}$ . 
\end{proof}

\begin{lemma}\label{l2}
Let $F$ be a projective curve that is given locally (in some neighbourhood $U$ of a generic point $p\in F$) by a family of holomorphic functions$\colon$ 
\[
f\colon U\rightarrow\mathbb{P}^n
\]
\[
z\mapsto[\ldots\colon f_{i}(z)\colon\ldots].
\]
If the dimension of the projective hull of $F$ equals $k$, then there exists a descending chain of pencils of osculating hyperplanes
$$
V^{0}\supset V^{1}\supset V^{2}\supset\ldots\supset V^{k},
$$
such that hyperplanes lying in $V^{i}$ intersect $F$ at p with multiplicity at least $i$, and $\dim V^{i}-\dim V^{i-1}=1$. In fact, $V^{i}$ contains a dense subset of hyperplanes that intersect $\gamma$ with multiplicity exactly $i$ and a proper subspace of hyperplanes that intersect $\gamma$ with multiplicity at least $i+1$ for $i<k$, and $V^{k}$ contains a unique hyperplane that intersects $\gamma$ with multiplicity $k$.
\end{lemma}
\begin{proof}[\underline{Proof}]
Let the dimension of the projective hull of $F$ be equal to $k$. Without loss of generality we may assume that $F$ lies in the projective space of dimension $k$ and is not contained in any hyperplane. The latter implies that coordinate functions $f_{i}(z)$ are linearly independent. It is well known that any set of analytic functions is linearly independent if and only if the Wronskian of this set is not identically 0. Recall that the Wronskian can be computed by the following formula
\[\det
\begin{bmatrix}
    f_{0}       & f_{1} & f_{2} & \dots & f_{k} \\
    f_{0}^{(1)}       & f_{1}^{(1)} & f_{2}^{(1)} & \dots & f_{k}^{(1)} \\
    \hdotsfor{5} \\
    f_{0}^{(k)}       & f_{1}^{(k)} & f_{2}^{(k)} & \dots & f_{k}^{(k)}
\end{bmatrix}.
\]
\\Now let us provide a relation between that property and the statement of the Lemma. Expand each $f_{i}$ as a Laurent power series at $p$ and substitute them into the expression $H(x_{0},\ldots, x_{k})=c_{0}x_{0}+c_{1}x_{1}+\ldots+c_{n}x_{n}$. Regrouping the terms by their exponents $z^i$ we obtain$\colon$
\[
H(f_{0}(z),\ldots, f_{k}(z))=z^{0}(c_{0}f_{0}(0)+\ldots+c_{k}f_{k}(0))/0!+\]
\[z^{1}(c_{0}f_{0}^{(1)}(0)+\ldots+c_{k}f_{k}^{(1)}(0))/1!+\ldots\]
\[
z^{j}(c_{0}f_{0}^{(j)}(0)+\ldots+c_{k}f_{k}^{(j)}(0))/j!+\ldots
\]
\[
z^{k}(c_{0}f_{0}^{(k)}(0)+\ldots+c_{k}f_{k}^{(k)}(0))/k!.
\]
The curve $F$ and the hyperplane $\{H=0\}$ intersect with multiplicity $i$ at $p$ if and only if\\ $\bar{H}(z)=H(f_{0}(z),\ldots, f_{k}(z))=cz^{i}+h.o.t., c\neq0$. We list the first $k+1$ coefficients of the expansion $\bar{H}(z)$ in the following matrix$\colon$
\[A=
\begin{bmatrix}
    f_{0}(0)       & f_{1}(0) & f_{2}(0) & \dots & f_{k}(0) \\
    f_{0}^{(1)}(0)       & f_{1}^{(1)}(0) & f_{2}^{(1)}(0) & \dots & f_{k}^{(1)}(0) \\
    \hdotsfor{5} \\
    f_{0}^{(k)}(0)       & f_{1}^{(k)}(0) & f_{2}^{(k)}(0) & \dots & f_{k}^{(k)}(0)
\end{bmatrix}.
\]
As we can see, the matrix $A$ is indeed the Wronskian of the system of functions $\{f_{i}\}$ at zero. Since $F$ does not lie in any hyperplane, the Wronskian matrix is not identically zero. Switching from $p$ to another point $q\in U$ if necessary, we conclude that $A$ has maximal rank. To finish the proof, we choose $V^{i}$ to be the set of hyperplanes corresponding to solutions of the first $i$ rows of $A$.
\end{proof}

Let $f$ be an irreducible polynomial supported at $B$ and $g$ be a polynomial supported at $A$. Assume that $g=fc$ for some $c\in\mathbb{C}[x^{\pm1},y^{\pm1}]$. For a given $f$ the set of all such $g$ is a finite-dimensional vector space over $\mathbb{C}$. We denote this vector space by $\mathcal{V}_{A}^{B}(f)$. It is clear that for two polynomials $f_{1}, f_{2}$ with the same support set $B$ we have $\mathcal{V}_{A}^{B}(f_1)\simeq\mathcal{V}_{A}^{B}(f_2)$ via the natural map $g\mapsto gf_{2}/f_1$. Thus, for generic $f$, $\dim\mathcal{V}^{B}_{A}(f)$ does not depend on the choice of $f$.
\begin{lemma}\label{l3}
For generic $f\in\mathbb{C}^{B}$ the following inequality holds$\colon\dim\mathcal{V}_{A}^{B}(f)\leq|\conv(A)\ominus B|$. If $A$ is convex, then the inequality turns into equality. 
\end{lemma}
\begin{proof}[\underline{Proof}]
Denote $\conv(A)\ominus B$ by $C$. It is well known from the theory of convex bodies that in the case $g=cf$ we have $c\in\mathbb{C}^{C}$ even if $\suppo(f)$ is not convex. Consider the linear map $$\alpha_{f}\colon\mathbb{C}^{C}\rightarrow\mathbb{C}^{\conv(A)}$$
$$
c\mapsto cf.
$$
Since $\alpha_{f}$ is an inclusion, we claim that $\dim\ima(\alpha_{f})=|C|$. Since $\mathcal{V}_{A}^{B}(f)\subset\ima(\alpha_{f})$ is a vector subspace consisting of the polynomials with zero coefficients at the monomials in $\conv(A)\setminus A$, we have $\dim\mathcal{V}_{A}^{B}(f)\leq\ima(\alpha_{f})=|C|$. If $A$ is a convex set, then we obtain the equality.
\end{proof}

Finally we are ready to prove Theorem \ref{th}

\begin{proof}[\underline{Proof of Theorem \ref{th}}]\label{pr}
Let $\gamma=\mathcal{Z}(f)$ be a smooth irreducible curve supported at $B$. We denote its image under the monomial map $m_{A}$ by $F$. The space of all hyperplanes containing $F$ is identified with $\mathcal{V}_{A}^{B}(f)$ by Lemma \ref{l1}, therefore, by Lemma \ref{l3}, the dimension of the projective hull of $F$ is $D=|A|-\dim(\mathcal{V}_{A}^{B}(f))-1\geq|A|-|\conv(A)\ominus B|-1$. 

Consider a generic (non-Weierstrass) point $p$ of the curve $\gamma$. By Lemma \ref{l2}, there exists a tuple of osculating hyperplanes $\{H_{i}\}_{i=0}^{D}$ such that $H_{i}$ intersects $F$ at $m_{A}(p)$ with multiplicity $i$. We denote the coefficients of these hyperplanes by $H_{a}^{i}$. Since $\gamma$ is a smooth curve, we can choose a holomorphic parametrization in some neighbourhood of $p\colon x=x(t), y=y(t)$. For each $i\in\{0,\ldots,D\}$ define $g_{i}=\sum_{a}H_{a}^{i}z^a\in\mathbb{C}^{A}$, $G_{i}=\mathcal{Z}(g_{i})$. Substitute $x(t), y(t)$ into the equation of $g$, and write the Laurent series for $g$
$$
g=g_{n}t^n+h.o.t.
$$
where $n$ is the intersection multiplicity of $\gamma$ and $G_{i}$ at $p$. Let $L$ be a linear form not vanishing at $m_{A}(p)$. Then $\sum_{a}H_{a}^{i}z_a/L$ is a meromorphic function that has the same order at $m_{A}(p)$ as $g$ has at $p$. Indeed, substituting $x(t), y(t)$ into the function $\sum_{a}H_{a}^{i}z_a/L$ we will obtain $(g_{n}t^n+h.o.t.)/(c_{0}+c_{1}t^{l}+h.o.t.), c_{0}, c_{1}\neq0, l>0$. Thus we have the Laurent expansion of the following form
$$
(g_{n}t^n+h.o.t.)/(c_{0}+c_{1}t^{l}+h.o.t.)=c^{-1}_{0}(g_{n}t^n+h.o.t.)(1-c_{1}c^{-1}_{0}t^{l}+h.o.t.)=c^{-1}_{0}g_{n}t^n+h.o.t.
$$
Therefore, the intersection multiplicities of $\gamma$ and $G_{i}$ at $p$ and $F$ and $H^{i}$ at $m_{A}(p)$ are both equal to $i$ as desired.
\end{proof}

Using the arguments as in the proof of Theorem \ref{th}, we can obtain the following rough generalization of Theorem $\ref{th}$.

\begin{theorem}\label{th3}
Let $(A,B)$ be a pair of subsets of $\mathbb{Z}^2$ such that neither $A$ nor $B$ lies on a segment and cannot be shifted to a proper sublattice, and let $\gamma$ be a generic curve supported at $B$. Consider a tuple $\{m_{1}, m_{2},\ldots, m_{l}\}$ of natural numbers such that $m_{1}+m_{2}+\ldots+m_{l}=|A|-|\conv(A)\ominus B|-1$. For generic points $\{p_1, p_2, \ldots , p_l\}$ there exists an osculating curve $\{g=0\}$ supported at $A$ such that $\gamma$ and $\{g=0\}$ intersect at $p_i$ with multiplicity at least $m_i$.
\end{theorem}

Note that from Theorem \ref{th} we have $I_{A, B}\geq\min(|A|-1, |B|-1)$, thus one might naturally expect that $I_{A, B, C}\geq\min(|A|-1, |B|-1, |C|-1)$ for $n=3$. The following remark shows that this expectation is actually not the case.\label{wr}
\begin{rem}\label{ex5}
Consider a tuple $S=(A, B, C)$ of bodies in $\mathbb{Z}^3$, where $$A=\{(0,0,0),(1,0,0),(1,1,0),(0,1,0)\},$$ 
$$B=\{(0,0,0),(0,1,0),(0,1,1),(0,0,1)\},$$ 
$$C=\{(0,0,0),(1,0,0),(1,0,1),(0,0,1)\}.$$

According to the expected formula $I_{A, B, C}\geq\min(|A|-1, |B|-1, |C|-1)$, we have a root of multiplicity $|A|-1=|B|-1=|C|-1=3$, however the mixed volume of this tuple is $2$, which leads us to a contradiction.
\end{rem}

\section{\bf Curves supported at integral triangles}\label{s1}
We shall say that a subset of $\mathbb{Z}^2$ is a triangle if it consists of three non-collinear points. In this section, we classify all triangles $B$ such that curves supported at $B$ do not have an inflection point on the complex torus. This classification will be necessary for further computations.

\begin{lemma}
Any two curves supported at a given triangle are isomorphic via some transformation $x=\alpha\bar{x}, y=\beta\bar{y}$. In particular, any curve supported at a triangle is smooth.
\end{lemma}
\begin{proof}[\underline{Proof}]Assume we have two curves supported at the subset $\{p,q,r\}\subset\mathbb{Z}^2$. They are given by the following equations$\colon$
$$
ax^{p_{x}}y^{p_{y}}+bx^{q_{x}}y^{q_{y}}+cx^{r_{x}}y^{r_{y}}=0;
$$
$$
\alpha x^{p_{x}}y^{p_{y}}+\beta x^{q_{x}}y^{q_{y}}+\gamma x^{r_{x}}y^{r_{y}}=0.
$$
The following system $$x^{p_{x}-q_{x}}y^{p_{y}-q_{y}}=b\alpha/a\beta;$$ $$x^{p_{x}-r_{x}}y^{p_{y}-r_{y}}=c\alpha/a\gamma$$ is solvable in radicals. Let $(x_{0}, y_{0})$ be one of the solutions. Applying the transformation $x\mapsto x_{0}x$, $y\mapsto y_{0}y$ to the curve given by the first equation we obtain the second curve. This concludes the first part of the proof. 

Assume that $F\in\mathbb{C}^{B}$, $B=\{(p_{x}, p_{y}),(q_{x}, q_{y}),(r_{x}, r_{y})\}$ and the curve $F=0$ has a singular point on the complex torus, then the following system has a nonzero solution at $(1, 1)$.
    $$
    \begin{cases}
    F=0\\
    xF_{x}=0\\
    yF_{y}=0.\end{cases}
    $$
    The matrix of this system is 
    $$
    \begin{pmatrix}
1 & 1 & 1\\
p_{x} & q_{x} & r_{x}\\
p_{y} & q_{y} & r_{y}
\end{pmatrix}\sim
\begin{pmatrix}
1 & 1 & 1\\
0 & q_{x}-p_{x} & r_{x}-p_{x}\\
0 & q_{y}-p_{y} & r_{y}-p_{y}
\end{pmatrix}.
    $$
    Since the triangle is not degenerate, the minor in the right lower angle is non-vanishing and the only possible solution is trivial.
\end{proof}

Let us state the following well-known fact.
\begin{lemma}
Let $\gamma$ be a curve given by the polynomial equation $F=0$. Then a smooth point $p\in\gamma$ is an inflection point if and only if $\det(\Hess(p))=0$, where 
$$
\Hess=
\begin{pmatrix}
F_{xx} & F_{xy} & F_{x}\\
F_{xy} & F_{yy} & F_{y}\\
F_{x} & F_{y} & 0
\end{pmatrix}.
$$
\end{lemma}

We claim that if a curve supported at $B$ has an inflection point in the complex torus, then there exists a curve which is supported at the same set and has an inflection point at $(1,1)$. The latter is easy to show by applying the transformation $x\mapsto x_{0}\bar{x}$, $y\mapsto y_{0}\bar{y}$, where $(x_{0}, y_{0})$ is the inflection point of the initial curve. If $F\in\mathbb{C}^B, B=\{(p_x, p_y), (q_x, q_y), (r_x, r_y)\}$ then clearly $F$ has the following form $ax^{p_{x}}y^{p_{y}}+bx^{q_{x}}y^{q_{y}}+cx^{r_{x}}y^{r_{y}}$ for some $a, b, c\in\mathbb{C}^{*}$ and for the given $B$ the Hessian determinant depends on the coefficients $a, b, c$, without loss of generality we will put $c=1$. For further computations, we denote by $\He(a)$ the Hessian determinant after the substitution $b=-1-a, c=1$, $x=1$, $y=1$.  

\begin{lemma}\label{l8}
If $B=\{(0,0),(n,m),(k,l)\}$ has no horizontal, vertical or anti-diagonal segments as a triangle, then $$\He(0),\He(-1)\neq 0.$$
\end{lemma}
\begin{proof}[\underline{Proof}]
Indeed, the direct computations imply that $\He(0)=-k l (k + l)$ and $\He(-1)=-m n (m + n)$. 
\end{proof}

\begin{rem}
The monomial change $x=uv^{-1}, y=v^{-1}$ is projective and, in particular, maps lines to lines, and therefore preserves inflection points.
\end{rem}

\begin{rem}
We will say that a tuple $S$ of subsets of $\mathbb{Z}^2$ has a root/inflection point or singularity of  prescribed type if there exists a system with the same property supported at $S$.
\end{rem}

Now let us investigate what types of triangles have inflection points. The previous results allow us to focus on the following cases$\colon$
\begin{enumerate}
    \item $B=\{(0,0),(n,m),(k,l)\}$, no segments are vertical, horizontal or in $(1,-1)$ direction.
    \item $B={(0,m),(n,0),(k,0)}$, $k\neq0$, $n\neq0$. 
\end{enumerate}
\begin{lemma}\label{lem4}
Up to projective monomial change of variables, i.e. $x=uv^{-1}, y=v^{-1}$, triangles that have no inflection points are listed below (see Figure \ref{f3})$\colon$
\begin{enumerate}
    \item $\{(0,0),(1,0),(0,n)\}$, $n\neq0$;
    \item $\{(n,0),(1-n,0),(0,1)\}$;
    \item $\{(1,0),(n,0),(0,1)$, $n\neq1$;
    \item $\{(0,0),(n,0),(0,n)\}$, $n\neq0$.
\end{enumerate}
\end{lemma}
\begin{proof}[\underline{Proof}]
\begin{enumerate}
    \item Let $F=1+ax^n y^m-(1+a)x^l y^k$, then direct computations imply that the leading coefficient of $\He(a)$ equals $(l - m) (k - n) (k + l - m - n)$ and, therefore it never vanishes in the initial conditions. Therefore, in this case $\He(a)$ is a non-constant polynomial and thus it has a root. By Lemma \ref{l8} this root is not equal to $0$ or $-1$ and the curve $\{F=0\}$ has an inflection point.
    \item Let $F = x^n + a x^m - (1 + a) y^k$. In that case $$\He(a)=(1 + a) k ((k-1) (a m + n)^2 - (1 + 
      a) k (a (m-1) m + (n-1) n)).$$ We may ignore $k(a+1)$ and consider terms in brackets. Let us introduce the notation $\Te(a)=(k-1) (a m + n)^2 - (1 + a) k (a (m-1) m + (n-1) n)$. A trinomial curve has no inflection points if and only if $\Te(a)$ has no roots except $a=0$ and $a=-1$. We classify all the cases that satisfy this condition$\colon$\\
      
      $\deg(\Te(a))=2$, the leading coefficient $(k - m) m\neq0$
      \begin{enumerate}
          \item $\Te(0)=\Te^{\prime}(0)=0\Rightarrow B=\{(0,0),(1,0),(0,n)\}$ or $B=\{(n,0),(n-1,0),(0,n)\}$ (these cases are affinely equivalent),
          \item $\Te(-1)=\Te^{\prime}(-1)=0\Rightarrow B=\{(n,0),(1-n,0),(0,1)\}$,
          \item $\Te(0)=\Te(-1)=0\Rightarrow B=\{(0,0),(n,0),(0,1)$ or $B=\{(1,0),(n,0),(0,1)$;
      \end{enumerate}
      
      $\deg(\Te(a))=1$, $m=0$
      \begin{enumerate}
          \item $\Te(0)=0\Rightarrow B=\{(0,0),(n,0),(0,n)\}$,
          \item $\Te(-1)=0$ gives no new configurations;
      \end{enumerate}
      
      $\deg(\Te(a))=1$, $m=k$
      \begin{enumerate}
          \item $\Te(0)=0$ gives no new configurations,
          \item $\Te(-1)=0$ gives no new configurations;
      \end{enumerate}
      
      $\deg(\Te(a))=0$ gives no new configurations.
      
\end{enumerate}
\end{proof}

\begin{figure}[h!]
  \centering
     \includegraphics[width=0.6\linewidth]{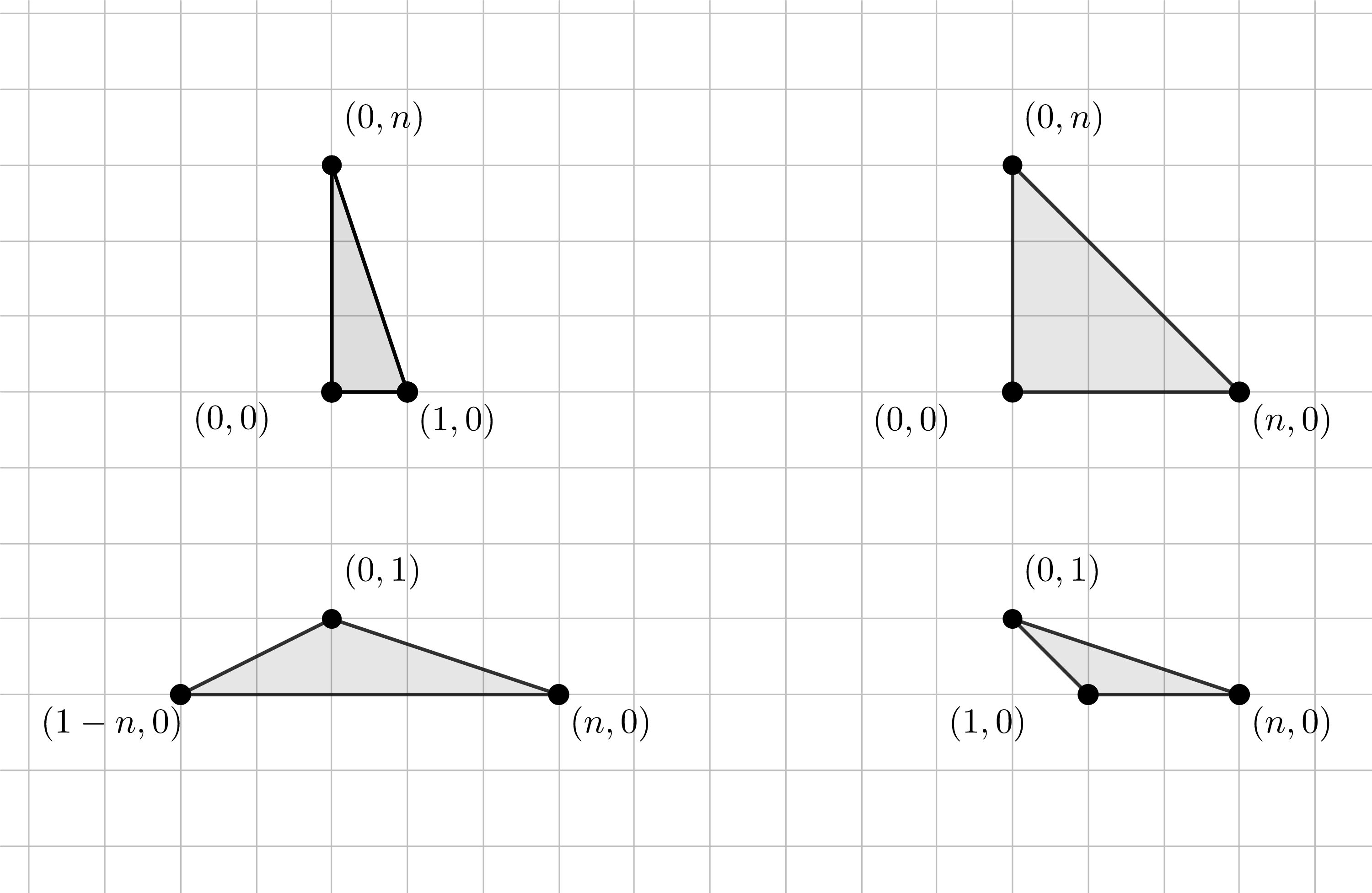}
    \caption{Triangles that have no inflection points}
\label{f3}
\end{figure}

\section{\bf Classification of systems with roots of low multiplicity}

\begin{definition}\label{d1}
Let $T=(A, B)$ be a tuple of Newton bodies. We will say that $T$ has a root of high multiplicity if there exists a system supported at $T$ that has a root of multiplicity higher than 2.
\end{definition}

\begin{definition}\label{d2}
Let $X$ be a finite convex subset of $\mathbb{Z}^2$. We will say that $X$  is contained in $n$ horizontal segments if there exist $n$ distinct horizontal lines $l_1, l_2,\ldots, l_n$, such that $X=\cup (X\cap l_i)$ and $X\cap l_i\neq\varnothing\ \forall i$.
\end{definition}

In this section we will classify all pairs of Newton polygons that do not have a root of high multiplicity. We will broadly use the following fact$\colon$ the property from Definition \ref{d1} of Newton bodies is invariant with respect to shifts and any invertible monomial changes of variables. For example, if a system have a multiple root then it will have a multiple root with the same multiplicity after any monomial change of variables. In this section we prove the following result

\begin{theorem}\label{th2}
A pair $(A,B)$ of Newton polygons does not have a root of high multiplicity if and only if up to monomial change of variables or transposition $(A,B)$ is one of the following (see Figure \ref{fim})$\colon$
\begin{enumerate}
\setcounter{enumi}{-1}
    \item The trivial case $A=\{(0, 0)\}$ and $B$ is arbitrary.
    \item $A=\{(0, 0), (1, 0)\}$, $B$ is contained in $3$ horizontal segments or $A=\{(0, 0), (1, 0), (2, 0)\}$, $B$ is contained in $2$ horizontal segments; Description of all sets that are contained in $2$ or $3$ horizontal segments can be found in Lemma \ref{cl23} and in the figure \ref{f2cl24};
    \item $A=B$ (possibly up to a shift) are of full dimension, and they are one of the following\\
    $\{(0, 0), (1, 0), (0, 1), (1, 1)\}$, $\{(0, 0), (1, 0), (0, 1), (2, 0)\}$;
    \item $A$ is the standard simplex and $B$ is a subset of 2-simplex.
    \end{enumerate}
\end{theorem}

\begin{figure}[h!]
  \centering
     \includegraphics[width=0.8\linewidth]{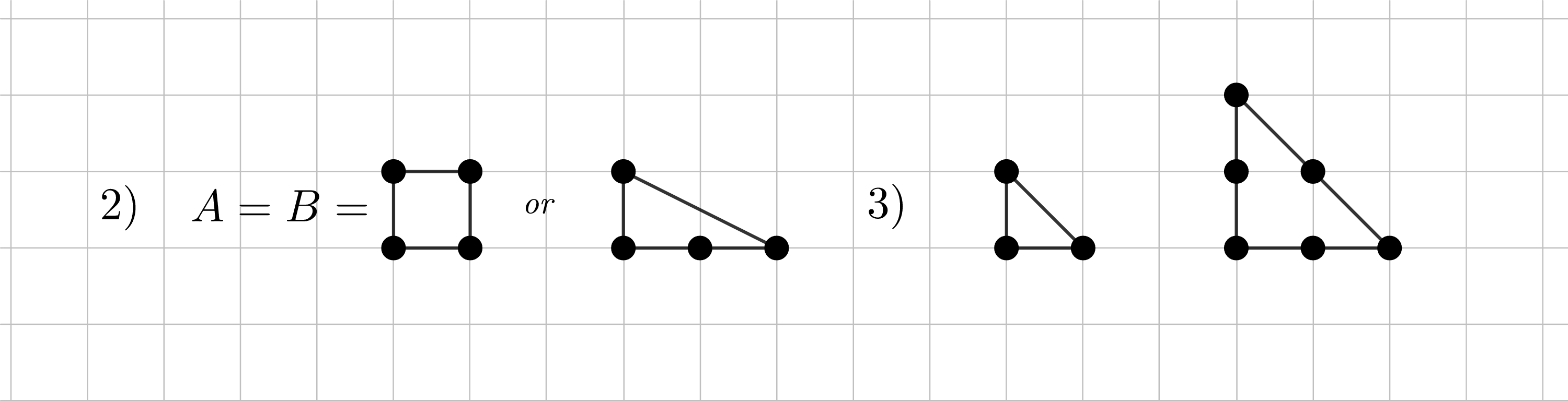}
    \caption{Pairs of full-dimensional Newton bodies without roots of multiplicities higher than 2}
\label{fim}
\end{figure}

First we describe geometric constraints for $(A, B)$.

\begin{lemma}\label{lm2}
If a pair $(A, B)$ of Newton polygons does not have a root of high multiplicity then one of the following conditions holds 
\begin{enumerate}
    \item $A$ or $B$ is a segment;
    \item $A=B$ and $|A|\leq 4$;
    \item $A\neq B$ and w.l.o.g. $|A|\leq 3$ and $|B|-|B\ominus A|\leq 3$.
\end{enumerate}
\end{lemma}

\begin{proof}[\underline{Proof}]
If neither $A$ nor $B$ is a segment then the conditions of Theorem \ref{th} are satisfied since $A, B$ are full-dimensional. In the case
$A=B$ there exists a system with root of multiplicity $|A|-1-1$ and since $(A, A)$ does not have a root of high multiplicity, we have $|A|-2\leq 2$. In the second case, assume that $A\ominus B$ and $B\ominus A$ are not empty. Then, there exist shift vectors $a, b$ such that $A+a\subset B$ and $B+b\subset A$ therefore $B+b+a\subset A+a\subset B$ and we obtain $a+b=0$ and $a=-b$. Then $A+a\subset B$ and $\{B-a\subset A\}\Rightarrow \{B\subset A+a\}\Rightarrow \{B=A+a\}$ which yields a contradiction. We obtain that one of them, say $B$, cannot be shifted to a subset of $A$ then $A\ominus B=\varnothing$ and system  has a root of multiplicity $|B|-|B\ominus A|-1\leq 2$ or $|A|-1\leq 2$. This concludes the proof.
\end{proof}

Now using Lemma \ref{lm2} we consider all the cases when the tuple $(A, B)$ does not have a root of high multiplicity. The first case is when $A$ is a segment. Without loss of generality we assume that $A$ is a horizontal segment passing through the origin. 

\begin{lemma}\label{lm222}
If $A$ is a horizontal segment of $h$ points and $B$ is contained in $v$ horizontal segments, then $(A, B)$ has no root of high multiplicity if and only if $(h, v)$ coincides with one of the following pairs
\begin{enumerate}
    \item $(n,1)$, $n\in\mathbb{N}$;
    \item $(1,n)$, $n\in\mathbb{N}$;
    \item $(3, 2)$, $(2,3)$, $(2,2)$.
\end{enumerate}
\end{lemma}
\begin{proof}[\underline{Proof}]
If $v=1$ or $h=1$, then the system $f=g=0, f\in\mathbb{C}^{A}$ has no isolated roots. If $v\geq 4$, $h\geq2$ then by Theorem \ref{ts} there exists a polynomial $p\in\mathbb{C}^{B}$ such that $p(1,y)$ has a root $y$ of multiplicity $3$. Then the system $g=f=0$ has the root $(1,y)$ of multiplicity $3$. If $h\geq 4$, $v\geq2$ then there exists a polynomial $f\in\mathbb{C}^{A}$ with root of multiplicity $3$, hence there exists a system $f=g=0$ with root of multiplicity 3. We obtain that if the tuple $(A, B)$ has no root of high multiplicity, then $h, v\leq 3$, which yields the following cases$\colon$ 
\begin{enumerate}
    \item $h=3$. The only possible case when $f\in\mathbb{C}^A$ and $g\in\mathbb{C}^B$ have roots of multiplicity $2$. Assume that $f$ has root $x=1$ of multiplicity $2$. Then consider $g=a+by+cy^2\in\mathbb{C}^{B}$ such that $a=-c (p-q)/p,b=-c q/p$. Then if $v=3$ the system $f=g=0$ has a root of multiplicity 4. 
    \item $h=2$. Assume that system has a root $(1,y)$, then the system has a root of multiplicity $3$ if and only if $g=a+by+cy^2=0$ has a root of multiplicity $3$. But this is impossible.
\end{enumerate}

\end{proof}

Now, let us classify all $B\subset\mathbb{Z}^2$ such that $B$ is contained in at most $3$ horizontal segments.

\begin{lemma}\label{cl23}
Up to affine transformations of $\mathbb{Z}^2$, preserving the horizontal direction, a convex subset $X\subset\mathbb{Z}^2$ is contained in at most 3 horizontal segments in one of the following cases$\colon$ 
\begin{enumerate}
    \item Sets, contained in 1 horizontal segment$\colon X=\{(0, 0), (1, 0),\ldots, (m, 0)\}, m\in\mathbb{N}\cup\{0\}$.
    \item Sets, contained 2 horizontal segments$\colon$
    \begin{enumerate}
        \item $X=\{(0,0), (k, l)\}$ with $k,l$ relatively prime.
        \item $X$ is the union of two arbitrary horizontal segments, first with the vertical coordinate $0$ and the second with the vertical coordinate $1$.
    \end{enumerate}    
    \item Sets, contained in 3 horizontal segments$\colon$
        \begin{enumerate}
        \item $X$ is the convex hull of a union of three arbitrary convex horizontal subsets of $\mathbb{Z}^2$, first with the vertical coordinate $0$, second with the vertical coordinate $1$ and the third with the vertical coordinate $2$. 
         \item $X=\{(0, 0), (1, 0), (2, 3), (1, 1)\}$
         \item $X$ is a primitive triangle with no horizontal edges or a non-horizontal segment of lattice length 2. 
    \end{enumerate}
        
\end{enumerate}
\end{lemma}
\begin{proof}
\begin{enumerate}
\item Assume $X$ is contained in 1 horizontal segment. This case is trivial.
    \item 
Assume $X$ is contained in $2$ horizontal segments. Let the height, lower base and the upper base be equal to $q, \delta_1, \delta_2$ respectively, (see figure \ref{f1cl23} case 1). The area of $X$ is given by the familiar formula $S_X=q(\delta_1+\delta_2)/2$ and from the different point of view by the Pick's formula $S_X=-1+0+(\delta_1+1+\delta_2+1)/2$. From this we obtain $(\delta_1+\delta_2)(q-1)=0$. Thus, in this case $\delta_1=\delta_2=0$ and $X$ is a segment $\{(0,0), (k, l)\}$ with $k,l$ relatively prime, or $q=1$ and $\delta_1$, $\delta_2$ are arbitrary.

\item Assume $X$ is contained in $3$ horizontal segments. We will classify all such sets by the number of points on the middle segment of $X$. If the middle segment contains at least two points, then there exist two primitive triangles (punctured triangles on figure \ref{f1cl23} case 2). Their areas are $q_{1}/2$ and $q_{2}/2$ but, by Pick's formula the area of a primitive triangle equals $1/2$, consequently $q_1=q_2=1$, in this case. If the middle segment consists of one interior point (see figure \ref{f1cl23} case 3), then by Pick's formula $(q_1+q_2)(\delta_1+\delta_2)/2=-1+1+(\delta_1+1+\delta_2+1)/2$, thus we obtain the equation $q\delta=\delta+2$, where $\delta=\delta_1+\delta_2$, $q=q_1+q_2$. The only possible solutions are $q=2,\delta=2$ and $q=3,\delta=1$. Without loss of generality we can assume that $\delta_1\geq\delta_1$. Additionally, we can apply the following conditions $q_1\delta_1/2=-1+0+(\delta_1+1+1)/2\iff (q_1-1)\delta_1=0$, $q_2\delta_2/2=-1+0+(\delta_2+1+1)/2\iff (q_2-1)\delta_2=0$. Finally, we have two systems of equations$\colon$
\begin{multicols}{2}
$$
\begin{cases}
(q_1-1)\delta_1=0\\
(q_2-1)\delta_2=0\\
\delta_1+\delta_2=2\\
q_1+q_2=2
\end{cases}
$$

$$
\begin{cases}
(q_1-1)\delta_1=0\\
(q_2-1)\delta_2=0\\
\delta_1+\delta_2=1\\
q_1+q_2=3
\end{cases}
$$

\end{multicols}
Their solutions are$\colon \{q_1=q_2=1, \delta_1=\delta_2=1\}$, $\{q_1=q_2=1, \delta_1-2=\delta_2=0\}$ and $\{q_1-2=q_2-1=0, \delta_1=\delta_2-1=0\}$, $\{q_1-2=q_2-1=0, \delta_1-1=\delta_2=0\}$. Cases $\{q_1=1, q_2=2\}$, $\{q_1=2, q_2=1\}$ are actually equivalent, so we do not consider the first one. Simple geometrical argumentation also shows that the case $\{q_1=2, q_2=1, \delta_1=1, \delta_2=0\}$ is impossible. In the last case the middle segment consists of the unique boundary point (figure \ref{f1cl23} case 4). The area of the punctured triangle is $1/2$, then, by Pick's formula again $1/2+q\delta/2=-1+0+(\delta_1+1+\delta_2+1+1)/2$, thus $q\delta=\delta$. And in this final case $X$ is a primitive triangle or a segment consisting of 3 points.

\end{enumerate}

\end{proof}

\begin{figure}[h!]
  \centering
     \includegraphics[width=0.8\linewidth]{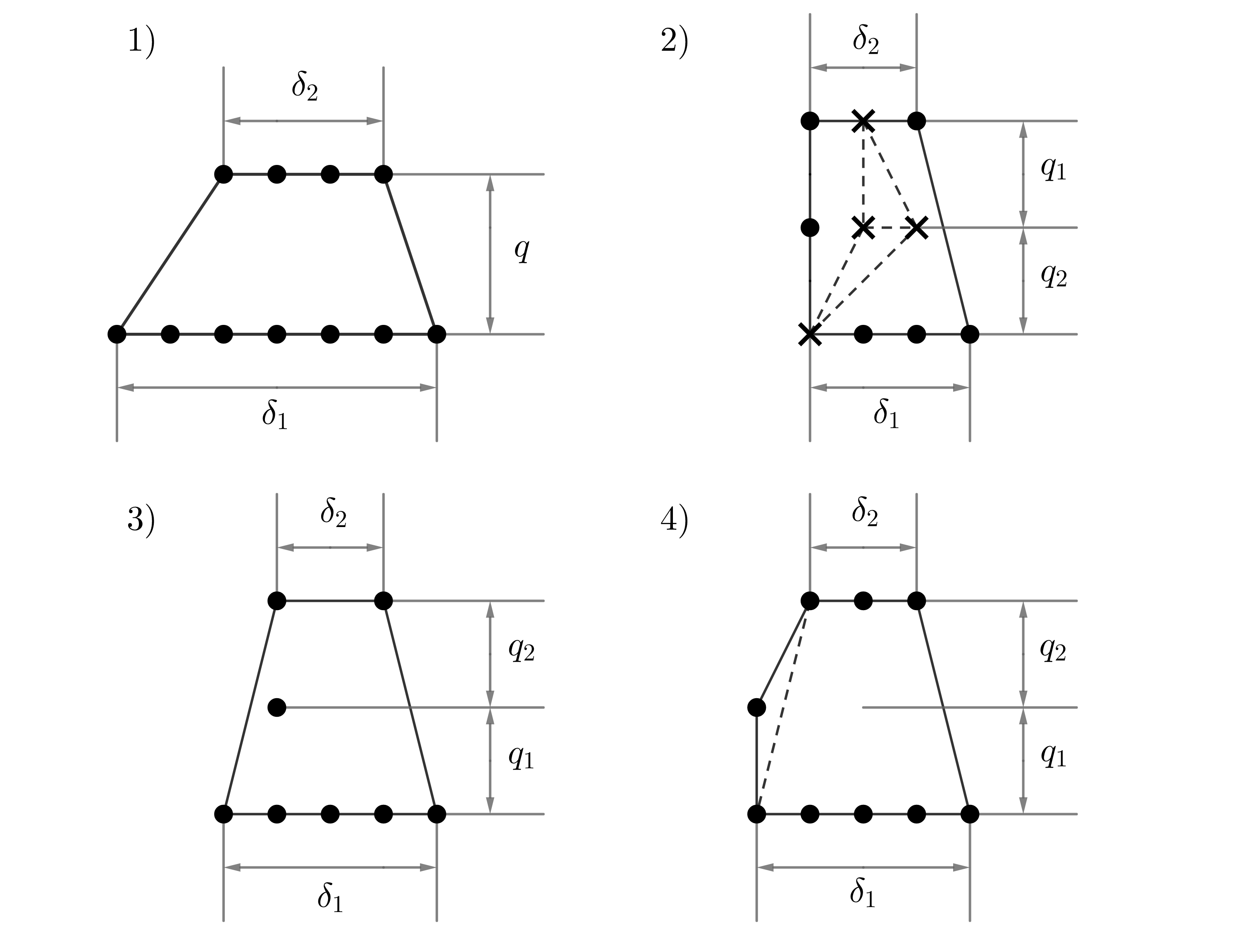}
    \caption{Illustration to the proof of lemma \ref{cl23}}
\label{f1cl23}
\end{figure}

\begin{figure}[h!]
  \centering
     \includegraphics[width=0.8\linewidth]{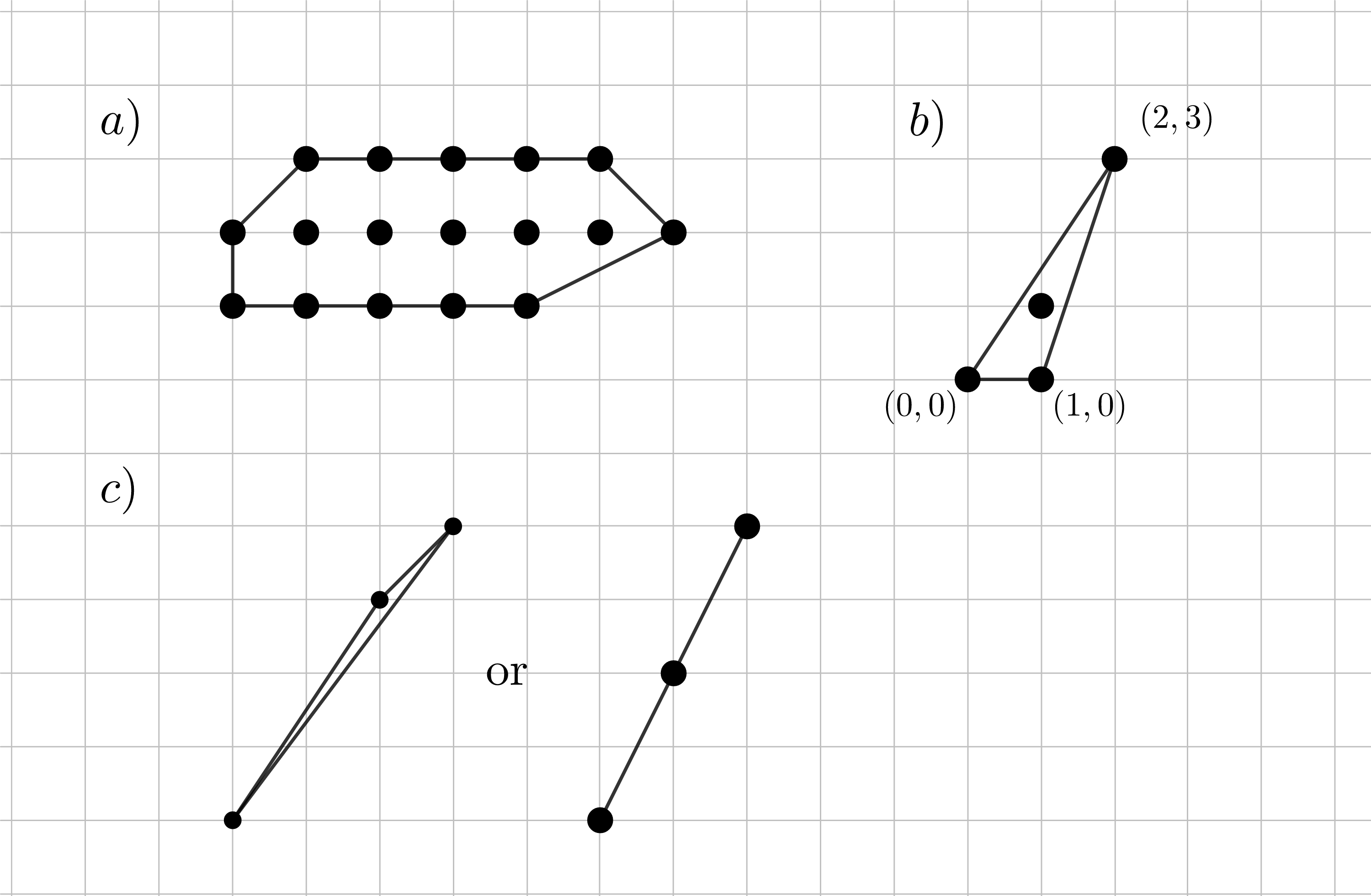}
    \caption{Sets those are contained in 3 horizontal segments}
\label{f2cl24}
\end{figure}

Now, we move to the case $A=B$. We have $|A\ominus B|=|A\ominus A|=1$. Theorem \ref{th} guarantees that the root of multiplicity $|A|-|A\ominus A|-1=|A|-2$ exists. Let us solve the following geometric inequality$\colon$
$$
|A|-2\leq 2.
$$
We have the following cases$\colon$
\begin{enumerate}
    \item Cases $|A|\leq3$ are trivial. 
    \item $|A|=4$. In non-degenerate case, there exists three points that do not lie on the same line. The forth point may lie in the interior or on the boundary, using the computations based on the Pick's formula we analyse all situations.
\end{enumerate}

\begin{lemma}\label{lm43}
If $|A|=4$, then, up to affine automorphisms of lattice, $A$ is one of the bodies displayed in the second part of Figure \ref{f1}.
\end{lemma}
\begin{proof}[\underline{Proof}]
We can choose 3 points that lie in the primitive triangle and make it the standard simplex via suitable change of variables. We denote the number of interior and boundary points of body by $i$ and $b$ respectively. Since $A$ is full-dimensional $b\geq 3$ and $i=0,1$. By Pick's formula, $S_{ABCV}=i/2+1$. Assume that the forth point $V$ is located on the distance $r$ from $BC$ on the line $l_{1}$, then the area of $ABCV$ satisfies the inequality $1/2+\sqrt{2}r_{1}/2\leq S_{ABCV}$. We consider two possible cases 
\begin{enumerate}
    \item $i=0$. $S_{ABCV}=1\Rightarrow r_{1}\leq 1, r_{2}\leq 1, r_{3}\leq 1$;
    \item $i=1$. $S_{ABCV}=3/2\Rightarrow r_{1}\leq \sqrt{2}, r_{2}\leq 2, r_{3}\leq 2$.
\end{enumerate}
As we can see, $V$ can be located only in the triangle bounded by the lines $l_{1}, l_{2}, l_{3}$. All possible options are depicted in the Figure \ref{f1}.  
\end{proof}

Bodies $\{(0,0),(1,0),(0,1),(-1,3)\}$ and $\{(0,0),(1,0),(0,1),(3,-1)\}$ are equivalent to\\ $\{(0,0),(1,0),(0,1),(-1,-1)\}$. A curve supported at $A=\{(0,0),(1,0)$, $(0,1),(-1,-1)\}$ is a cubic curve that has at least one inflection point on the complex torus. Hence we conclude that $(A, A)$ has a root of high multiplicity. Other bodies have mixed volume $2$ and then they have no roots of multiplicity higher than 2.

\begin{figure}[]
  \centering
  \begin{subfigure}[b]{0.7\linewidth}
    \includegraphics[width=\linewidth]{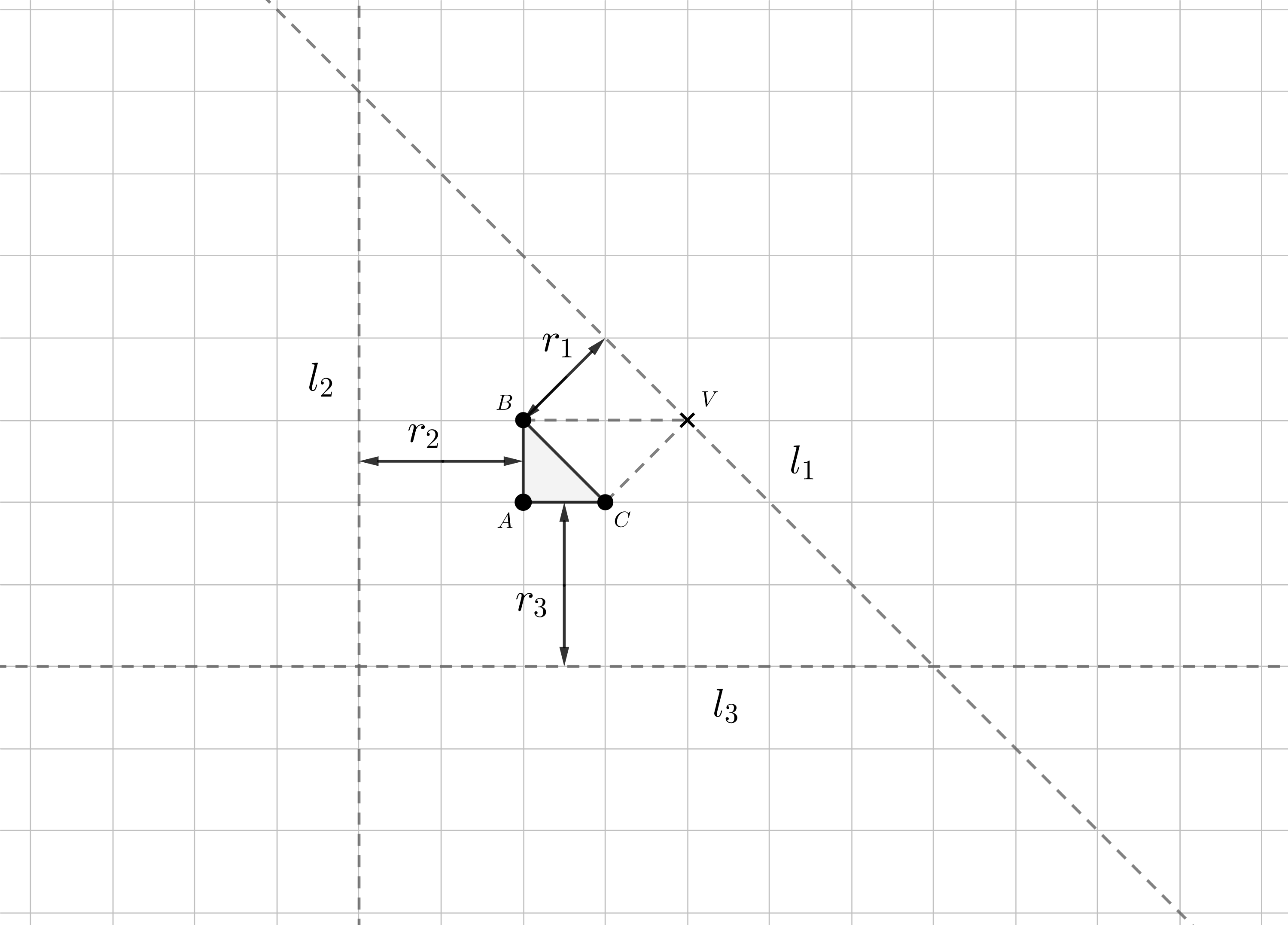}
    \caption{}
  \end{subfigure}
  \begin{subfigure}[b]{0.65\linewidth}
    \includegraphics[width=\linewidth]{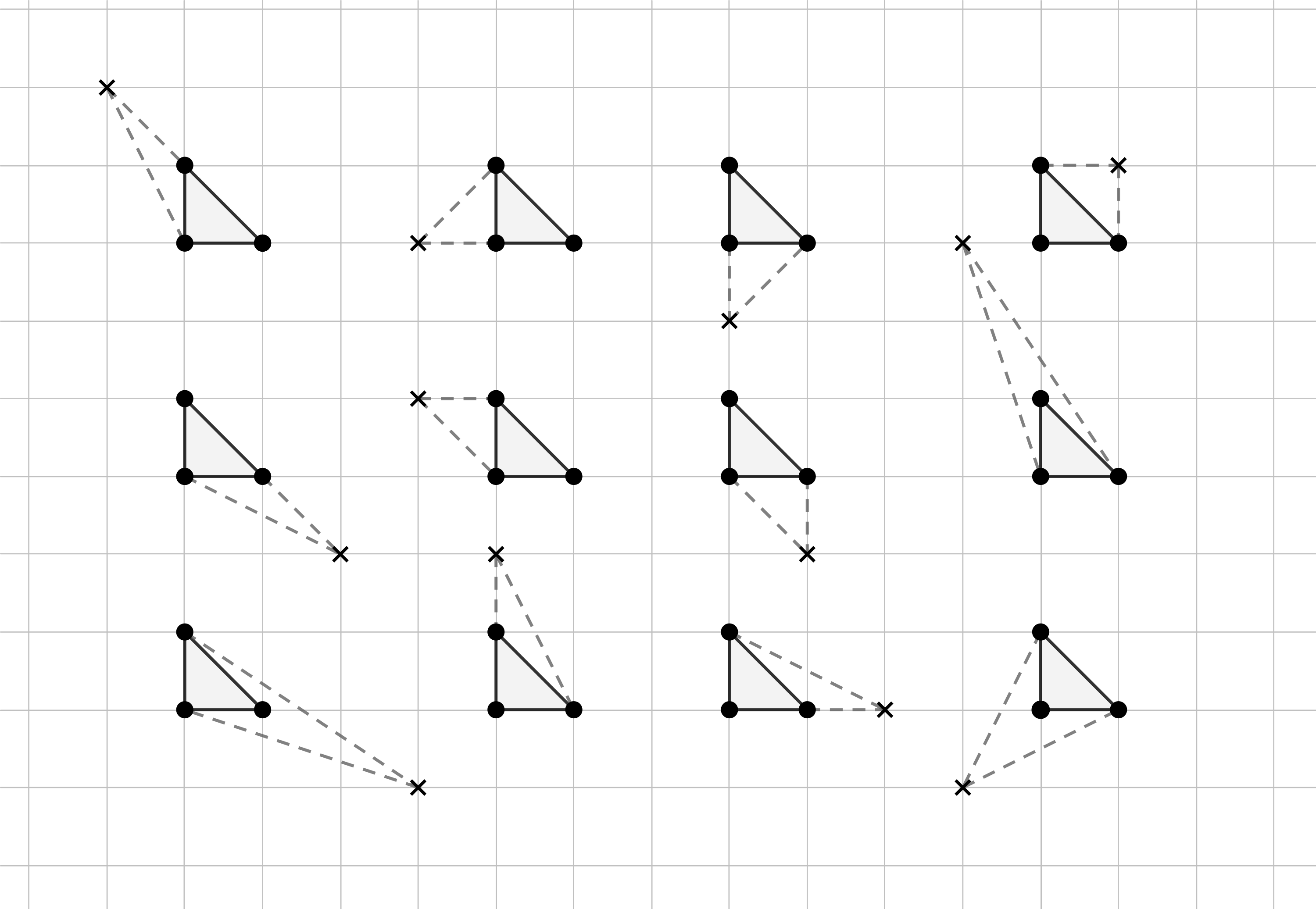}
    \caption{}
  \end{subfigure}
  \caption{Bodies that are obtained in the proof of Lemma \ref{lm43}.}
  \label{f1}
\end{figure}

We just proved the following statement$\colon$

\begin{lemma}
If $A$ is a full-dimensional body then $(A,A)$ has no roots of high multiplicity if and only if $A$ is the standard simplex, or $|A|=4$ and up to monomial change of variables $A$ is one of the bodies from the item 2 in Figure \ref{fim}. 
\end{lemma}

Now we are going to classify all Newton pairs $(A, B)$ such that $A$ and $B$ are not equal to each other and satisfy the conditions of Theorem \ref{th}. We will use two important technical facts$\colon$Lemma \ref{lem4} and Lemma \ref{lm2}.

Without loss of generality we will assume that $A\ominus B=\varnothing$. Applying Theorem \ref{th} to both pairs $(A, B)$ and $(B, A)$, we obtain the following two inequalities$\colon$

$$
|A|-|A\ominus B|-1=|A|-1\geq 3
$$
$$
|B|-|B\ominus A|-1\geq 3.
$$
Each of them is a sufficient condition on $(A, B)$ to have root of high multiplicity. We will solve the opposite inequalities to classify all pairs $(A, B)$ that do not have a root of high multiplicity.
$$
\begin{cases}
|A|-|A\ominus B|-1=|A|-1\leq2\\
|B|-|B\ominus A|-1\leq2.
\end{cases}
$$
From the inequalities we obtain $|A|\leq 3$, $|B|-|B\ominus A|\leq3$. Then, after suitable change of variables, $A=\{(0, 0), (1, 0), (0, 1)\}$. Since $A$ and $B$ are convex, then $B\ominus A$ is also convex, let us denote it by $C$. If $C$ is not empty then $A+C$ is a subset of $B$ and we get the inequality $|B|\geq |C+A|$, hence 
$$
|B|-|B\ominus A|\leq3\Rightarrow |C+A|-|C|\leq|B|-|B\ominus A|\leq 3\Rightarrow|C+A|-|C|\leq3. 
$$

The set $A$ can be reduced to the form $A=L_{x}\cup L_{y}$ where $L_{x}=\{(0, 0), (1, 0)\}$ and $L_{y}=\{(0, 0), (0, 1)\}$. We claim that$$ |C+A|=|C+L_{x}|+|C+L_{y}|-|(C+L_{x})\cap(C+L_{y})|,$$ and $$|C+L_{x}|=|C|+h+1,$$ where $h$ is the height of $C$ (see Figure \ref{f2}), and $|C+L_{y}|-|(C+L_{x})\cap(C+L_{y})|\geq 1$. Indeed, to obtain the last inequality consider a point $D$ of $C$ whose vertical coordinate is maximal and note that $D+(0, 1)$ cannot lie in $C+L_{x}$. Finally we deduce the following inequality$\colon$ 
$$
|C+A|-|C|\leq3\Rightarrow |C|+h+1+1-|C|\leq|C+A|-|C|\leq 3\Rightarrow h\leq1.
$$
Similarly, we get the inequality $l\leq1$, where $l$ is the horizontal length of $C$. Now, denote $|B|-|A+C|$ by $d$ and rewrite the inequality above$\colon$
$$
|B|-|B\ominus A|-1=|B|-|A+C|+|A+C|-|C|-1\leq2\Rightarrow d+|A+C|-|C|-1\leq 2\Rightarrow
$$
$$
\Rightarrow d+|C|+h+1+1-|C|\leq d+|A+C|-|C|\leq 3\Rightarrow
$$
$$
d+h\leq1.\ \text{And by the symmetry, we obtain the inequality}\ d+l\leq1.
$$
\begin{figure}[h!]
  \centering
     \includegraphics[width=0.6\linewidth]{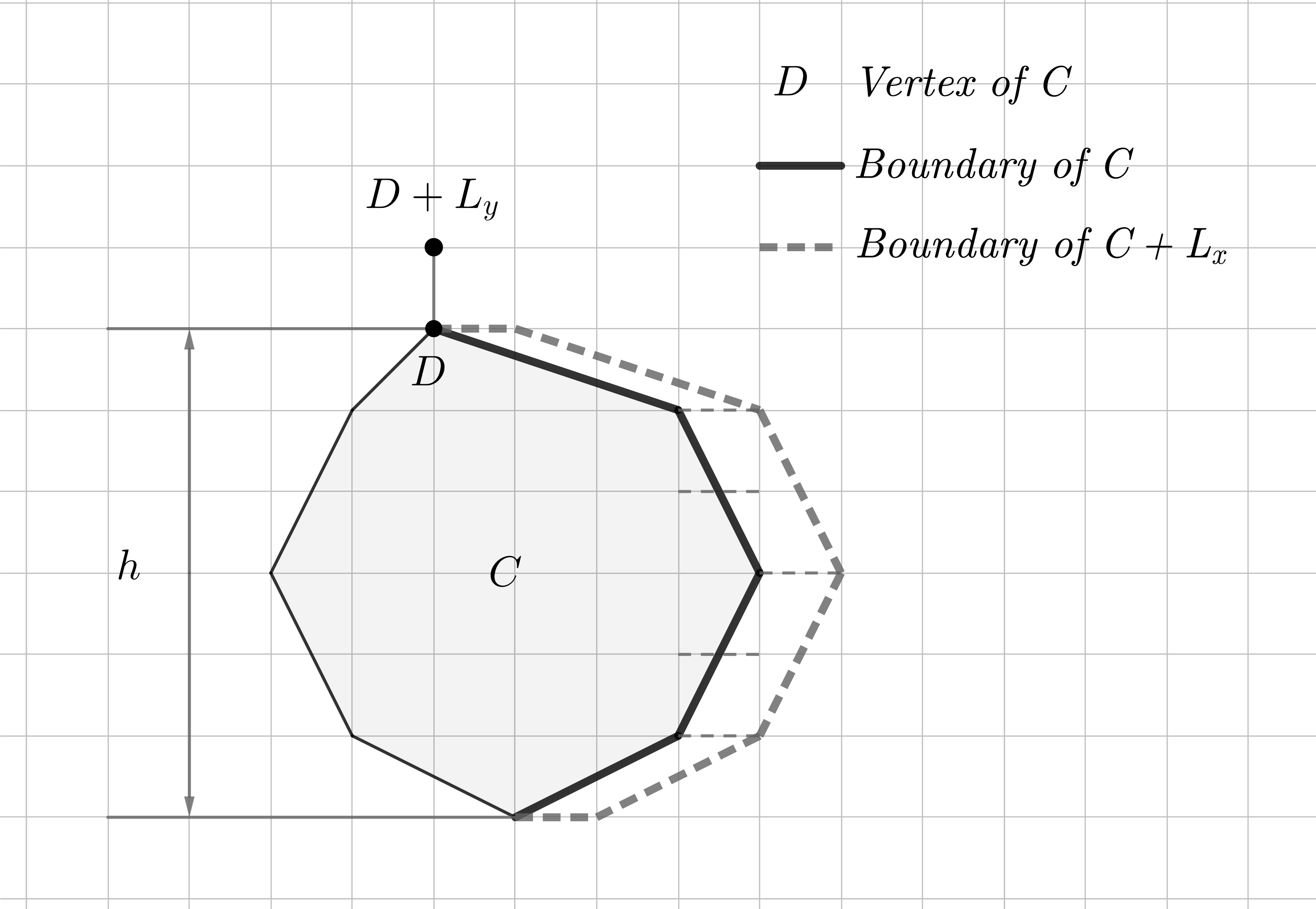}
    \caption{Illustration to the formula $|C+L_{x}|=|C|+h+1$}
\label{f2}

\end{figure}

Now let us analyse the obtained data. Since $l,h\leq 1$, $C=B\ominus A$ is a subset of the unit square, the possible cases are $|C|\in\{0, 1, 2, 3, 4\}$. Below we consider each of the cases.
\begin{enumerate}
    \item $C=\varnothing$, and hence $|B|=3$. This case has been already considered in section \ref{s1}. The only possible body with this property that does not have an inflection point is the standard simplex.

    \item $|C|=1\Rightarrow h=l=0$ then $d\leq 1$.
    \begin{enumerate}
        \item $d=0$ thus $A=B$, and curves supported at this pair are lines that can have only one intersection point,
        \item $d=1$ thus $A=B\cup\{p\}$ and $|A|=4$. We use Pick's formula to estimate the boundaries of $A$. The bodies that satisfy all the conditions are the same subsets of the 2-simplex as were displayed in Figure \ref{f1},

    \end{enumerate}

\item $|C|=2$ then $l\geq1$ or $h\geq1\Rightarrow d=0$. In this case $C$ is one of the following types$\colon\{(0, 0), (1, 0)\}$, $\{(0, 1), (1, 0)\}$, $\{(0, 0), (1, 1)\}$, and $B$ is one of the bodies depicted at Figure \ref{g24}. The last body has an inflection point since it contains $\{(0, 0), (1, 0), (0, 1), (-1, -1)\}$ as a subset,

 \begin{figure}[h!]
  \centering
     \includegraphics[width=0.5\linewidth]{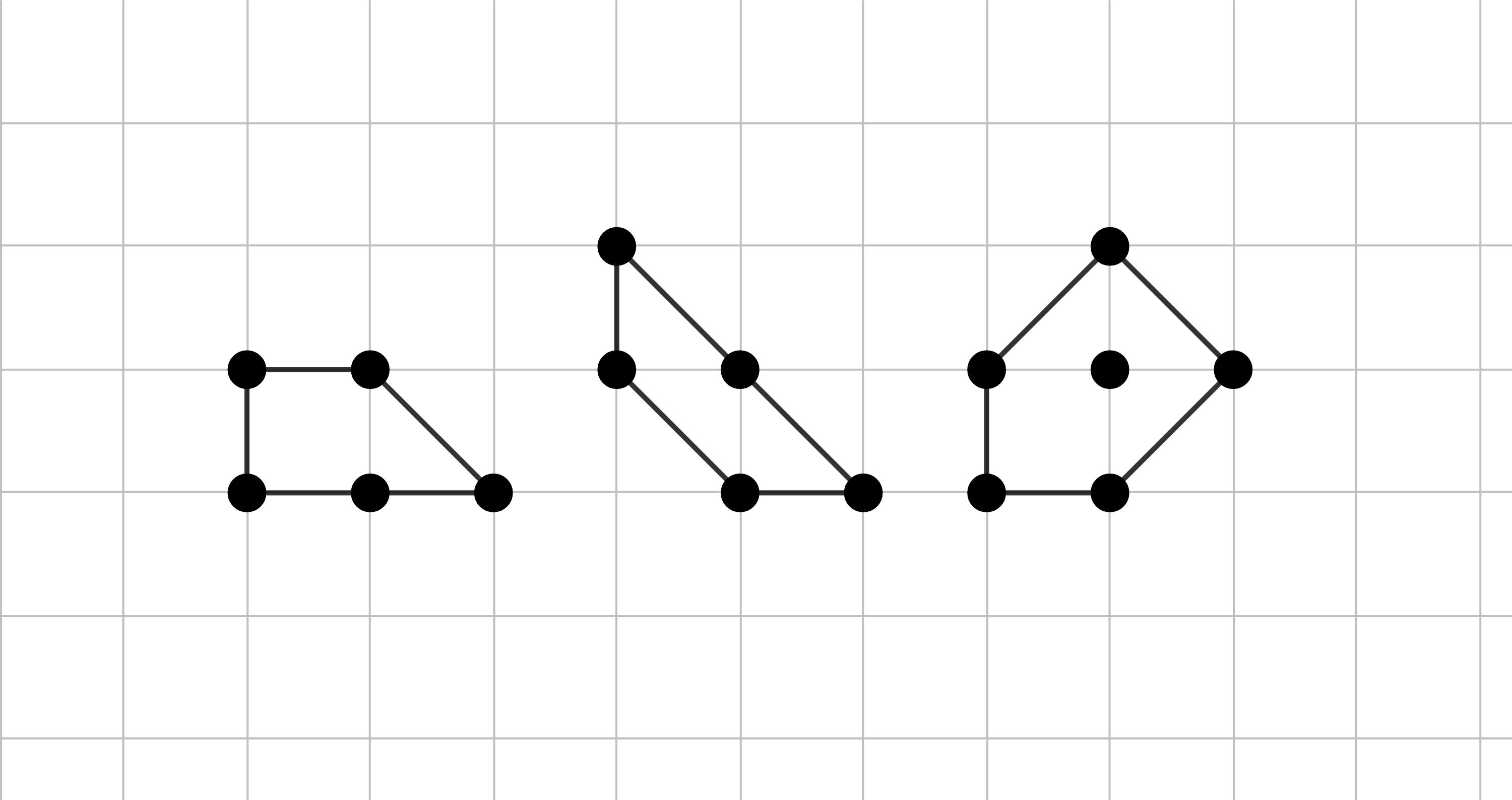}
    \caption{The case $|C|=2$}
\label{g24}
\end{figure}

\item $|C|=3$, $d=0$. In this case $C$ can be one of the following types$\colon$
\begin{multicols}{2}
\begin{itemize}
    \item $C={(0,0),(1,0),(0,1)}$, 
    \item $C={(0,0),(0,1),(1,1)}$,
    \item $C={(1,0),(1,1),(0,1)}$,
    \item $C={(0,0),(1,0),(1,1)}$.
\end{itemize}  
\end{multicols}
In this case all bodies $A+C$ are depicted at Figure \ref{f5}. 
\begin{figure}[h!]
  \centering
     \includegraphics[width=0.6\linewidth]{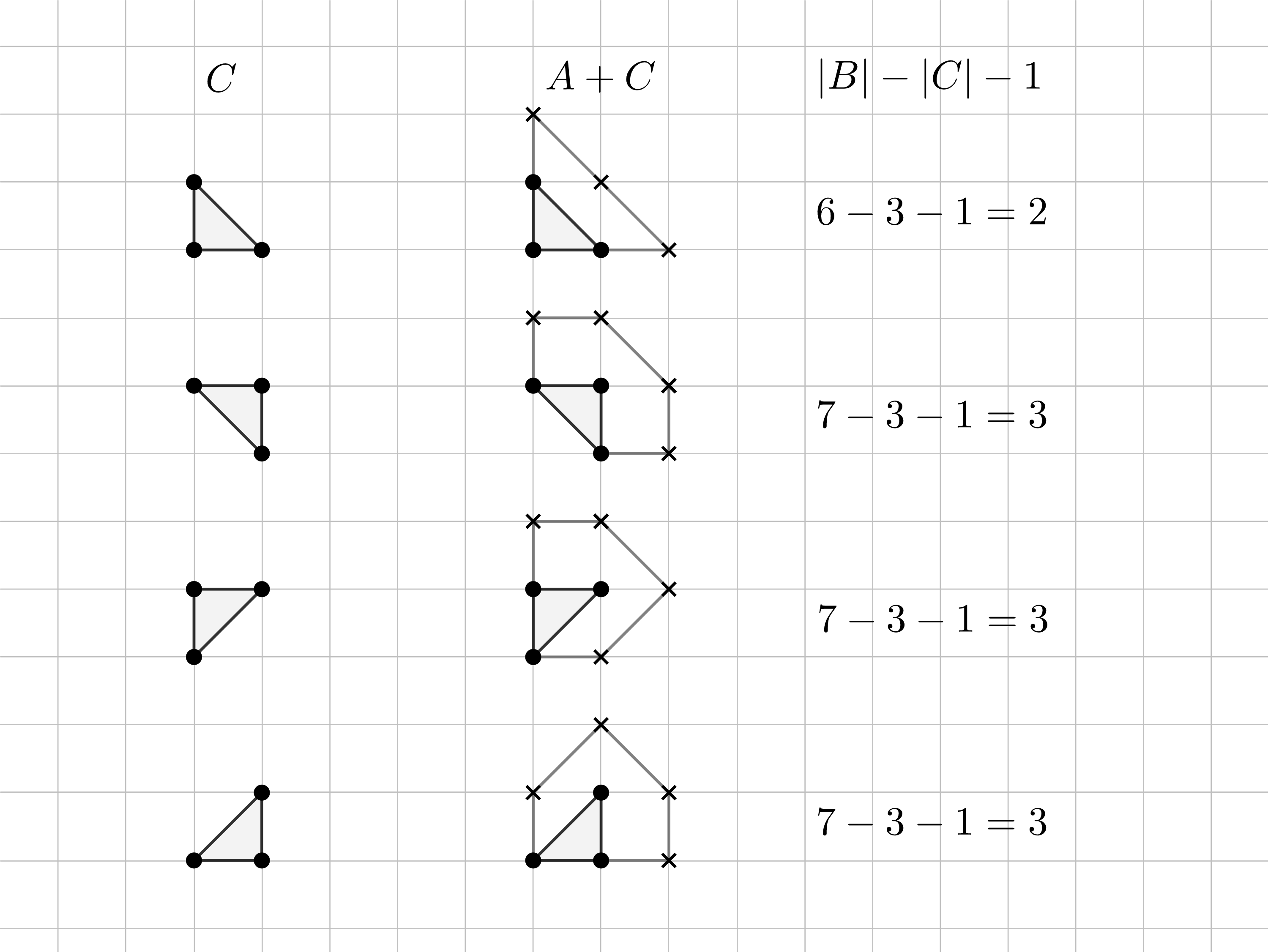}
    \caption{The case $|C|=3$}
\label{f5}
\end{figure}
And all of them except the standard simplex have a root of high multiplicity by Theorem \ref{th}. 
\item $|C|=4$, $C$ is the unit square and $d=0$. Consider $B=A+C$ and apply Theorem \ref{th}. We obtain that this configuration has an inflection point since $|B|-|B\ominus A|-1=|A+C|-|C|-1=8-4-1=3$. 
\end{enumerate}
This concludes the proof of Theorem \ref{th2}.


\section{Acknowledgments}

I would like to express my gratitude to Alexander Esterov for productive discussions and his great support. I would also like to thank Askold Khovanskii for important remarks. I would like to thank Arina Voorhaar for valuable comments and feedback. I am also grateful to reviewers for their deep and comprehensive analysis of the initial manuscript and pointing out some missing cases in Lemma \ref{cl23}.


\end{document}